\documentclass{amsart}
\usepackage[mathscr]{eucal}
\usepackage{amsmath,amssymb,hyperref}
\usepackage{amstext}
\usepackage{graphics}
\usepackage{xcolor}
\usepackage{mathrsfs}
\usepackage{epstopdf}

\usepackage{tikz}
\usetikzlibrary{matrix}

\usepackage[all]{xy}

\usepackage[T1]{fontenc}
\usepackage[utf8]{inputenc}

\newtheorem{thm}{Theorem}[section]
\newtheorem{THM}{Theorem}

\newtheorem{cor}[thm]{Corollary}
\newtheorem{prop}[thm]{Proposition}

\newtheorem{lemma}[thm]{Lemma}

\theoremstyle{definition}

\newtheorem{definition}[thm]{Definition}
\newtheorem{remark}[thm]{Remark}
\newtheorem{example}[thm]{Example}

\def\X0{X^{\circ}}

\def\Y0{Y^{\circ}}

\input xy
\xyoption{all}

\addtocounter{section}{0}             
\numberwithin{equation}{section}       

\begin{document}
\title[Positive characteristic Poincaré Lemma ]
{Positive characteristic Poincaré Lemma  }

\author[E. A. Santos]{Edileno de Almeida SANTOS}
\address{Instituto de Ciência, Engenharia e Tecnologia (ICET) - Universidade Federal dos Vales do Jequitinhonha e Mucuri (UFVJM), Campus do Mucuri, Teófilo Otoni - MG,
Brazil}
\email{edileno.santos@ufvjm.edu.br}

\author[S. Rodrigues]{Sergio RODRIGUES}
\address{Faculdade de Ciências Exatas e Tecnologia (FACET) - Universidade Federal da Grande Dourados (UFGD), Rodovia Dourados - Itahum, Km 12 - Cidade Universitátia, Dourados - MS,
Brazil}
\email{sergiorodrigues@ufgd.edu.br}


\begin{abstract}

Let $K$ be a field of characteristic $ p>0$ and $\omega$ be an $r$-form in $ K^n$. In this case, differently of fields of characteristic zero, the Poincaré Lemma is not true because there are closed $ r$-forms that are not exact. We present   here a definition of a $p$-closed $r$-forms and   a version of the Poincaré Lemma that is valid for $p$-closed  polynomial or rational $r$-forms on $ K^n$  and, as a consequence,  the de Rham cohomology  modules of $ K^n$ are not trivial.
\end{abstract}

\maketitle

\setcounter{tocdepth}{1}
\sloppy

\section{Introduction}
The Poincaré Lemma (\cite{Warner}, 4.18, p. 155) asserts that  if $ \omega$ is a differential $ r$-form ($ r>0$) on $\mathbb R^n$ or $\mathbb C^n$ (zero characteristic), then: $\omega$ is closed ($d\omega=0)$ if and only if $\omega$ is exact, that is, $\omega= d \eta$ for some $(r-1)$-form $\eta$. This is true in the differentiable category, but also can be done in the algebraic one if $K$ is of characteristic $0$.

If the field $K$ has characteristic $p>0$, there are closed $r$-forms on $ K^n$ that are not exact. M.E. Sweedler in \cite{Sweedler} has  defined the notion of {\it $ p$-closed} $1$-form on $ K^n$, and he proves that $\omega$ is $p$-closed if and if it is exact.


Finally, in Section \ref{S:Main}, we extend the definition of a $p$-closed form. We make the alternative definition of "closedness" for $r$-forms ("$p$-closedness", see Definition \ref{D:p-closed}), in such a way that it is possible to obtain the following algebraic general version of Poincaré Lemma bellow. 

\begin{THM}[Theorem \ref{T:Main}]
Let $K$ be a field of positive characteristic $p>0$ and $ R$ be $ K[z]$ or $ K(z)$. An $r$-form $\omega\in \Omega^r_{R/K}$ is exact if and only if $\omega$ is $p$-closed.
\end{THM}

In Section \ref{S:Decomposition}, we deduce from Poincaré Lemma the general {\it Inverse of Cartier Operator}.

\begin{THM}[Theorem \ref{T:Main2}]
$\gamma_0: \Omega_{K(z)/K}^*\rightarrow H_{K(z)}^*$ is an isomorphism.
\end{THM}

In fact, we will see also how can be followed the opposite direction in order to conclude the Poincaré Lemma from the Cartier Operator, as made in Section \ref{S:Cartier}.

\section{Positive Characteristic Differential Forms}

We introduce here the main notions about differential forms on $ K^n$, where $ K$ is a positive characteristic field. The  reader can see in \cite{Santos} and \cite{Santos2} the details of the exterior algebra of differential forms over $K^n$.

\begin{definition}
Let $K$ be a field of arbitrary characteristic $p$ ($p=0$ or $p$ is a prime integer). The {\it ring of differential constants} is the subring of $K[z]:=K[z_1,..., z_n]$ given by
$$
K[z^p]:=K[z_1^p,...,z_n^p]
$$
and the {\it field of differential constants} is the sub-field of $K(z):=K(z_1,..., z_n)$ given by
$$
K(z^p):=K(z_1^p,...,z_n^p)
$$
There are natural inclusions $K[z]\subset K(z)$ and $K[z^p]\subset K(z^p)$. Note that, in characteristic $0$, we have $K(z^p)=K(z^0)=K[z^p]=K[z^0]=K$. Otherwise, in prime characteristic $p>0$, we obtain $K[z^p]$ as a $K$-submodule of $K[z]$ generated by $\{G^p: G\in K[z] \}$ and $K(z^p)$ as the $K$-vector subspace of $K(z)$ generated by $\{g^p: g\in K(z) \}$.  The elements of $K(z^p)$ are called  {\it $\partial$-constants}, because $ d (g^p)=d(G^p)=0$.

\end{definition}

 

 
Here we will deal with $\Omega_{R/K}^r$, where $ R$ is the ring $K[z]$ or the field $K(z)$. The closedness or the exactness of an algebraic $ r$-form $ \omega$ is not changed by multiplication by a constant $\lambda \in K$ or a $\partial$-constant $\lambda \in K(z^p)$ and, by (iii) of the next proposition, one can always  suppose that $\omega$ is a polynomial $r$-form.

In what follows we consider multi-indexes $I=(i_1,...,i_n)$. If $J=(j_1,...,j_n)$ is another multi-index of {\it length} $n$, we write $J<I$ (resp. $J\leq I$) if $j_1<i_1$,..., $j_n<i_n$ (resp. $j_1\leq i_1$,..., $j_n\leq i_n$). We indicate: $\overrightarrow{0}=(0,...,0)$ and $\overrightarrow{p}=(p,...,p)$.

\begin{prop}[See also \cite{Brion}, 1.1.1 Lemma, p. 3]\label{preparation1}
Let $ K$ be a field of characteristic $p>0$.
\begin{itemize}
\item[(i)] $K[z]$ is a finite $K[z^p]$-module and $ K(z)$ is a finite dimensional $K(z^p)$-vector space, and $\dim_{K(z^p)}(K(z))=p^n$.
\item[(ii)]
$K(z^p)=\{f: f\in K(z): df=0 \}$  and  $ K[z^p]=\{F: F\in K[z]: dF=0 \}$
\item[(iii)] If $ \omega \in \Omega_{K(z)/K}^r$ is a rational $r$-form, then there is a $ \partial$-constant $ \lambda \in K(z^p)$ such that $ \lambda \omega$ is a polynomial $r$-form. 
\end{itemize}
 \end{prop}

\begin{proof}
\item[(i)]
A base for $ K[z]$ over $K[z^p]$ is given by the monomials $z_1^{i_1}...z_n^{i_n}$, where $0\leq i_1,...,i_n\leq p-1$ and this is also a base  for $ K(z)$ over $K(z^p)$. In fact, if $f=\frac{P}{Q}\in K(z)$, we can write
$$
f=\frac{P}{Q}=\left(\frac{1}{Q}\right)^p\cdot Q^{p-1}P
$$
and the polynomial $Q^{p-1}P$ can be expressed as a $K[z^p]$-linear combination of the monomials $z_1^{i_1}...z_n^{i_n}$.

\item[(ii)] Let $f\in K(z)$ be a rational function. As we have show
$$
f(z)=\sum_{\overrightarrow{0}\leq I <\overrightarrow{p}} g_I(z^p)z^I
$$
Hence, writing $I=(i_{(I,1)},...,i_{(I,n)})$,
$$
\frac{\partial f}{\partial z_j}=\sum_{i_{(I,j)}>0} g_I(z^p)i_{(I,j)}z^Iz_i^{-1}
$$
Since the monomials $i_{(I,j)}z^Iz_i^{-1}$ are $K(z^p)$-linearly independent, if $df=0$, then all $g_I$ with $\overrightarrow{0}< I$ must be $0$, so we must have $f=g_{\overrightarrow{0}}(z^p)\in K(z^p)$.

\item[(iii)] If $\omega$ is a rational $r$-form, then there is a polynomial $ Q \in K[z]$ that clears all the denominators of $\omega$, that is,  $ Q \omega$ is polynomial and then $Q^p \omega$ is polynomial. 

\end{proof}

By part (ii) of the above proposition, $ \Omega_{K(z)/K}^r $ is isomorphic to $ \Omega_{K(z)/K(z^p)}^r$. Similarly $ \Omega_{K[z]/K[z^p]}^r$ is isomorphic to $\Omega_{K[z]/K}^r$.

\section{Positive Characteristic Poincaré Lemma }\label{S:Main}

If $K$ is of characteristic 0 (for example  $K=\mathbb R$ or $K=\mathbb C$) it is well known that  a polynomial $r$-form is closed if and only if it is exact, that is, $d\omega=0$ if and only if $\omega=d\eta$ (Poincaré's Lemma). If $ K$ is of  characteristic $p>0 $, then there exist closed forms that are not exact: for example, $\eta= z_i^{p-1} dz_i \in \Omega^1_{K(z^p)/K} $ is closed, but it is not exact because the "integration $\int z_i^{p-1}dz_i $" does not results a polynomial. In this section we will develop a similar result of Poincaré Lemma over fields of arbitrary characteristic.

The $d$ operator on $r$-forms over $R$, where $R=K[z]$ or $R=K(z)$, gives us a chain sequence:
$$
0\longrightarrow K\hookrightarrow R\overset{d}{\longrightarrow} \Omega^1_{R/K}\overset{d}\longrightarrow ...\overset{d}\longrightarrow\Omega^n_{R/K}\longrightarrow 0
$$

If $K$ has characteristic $0$ and $R=K[z]$, this sequence is proved to be exact (see \cite{Hartshorne}, {\it Proposition (7.1) (Poincaré lemma)}, page 53), that is, an $r$-form is exact if and only if it is closed. In order to give a positive characteristic version of this important theorem, we need more than simply the closedness of forms. The following definition will provide a necessary and sufficient condition for exactness.

\begin{definition}\label{D:p-closed}
Let $K$ be a field of positive characteristic $p>0$. An $r$-form $\omega=\sum_{I} a_I dz_I\in \Omega_{R/K}^r$ is {\it $p$-closed} if $\omega$ is closed (that is, $d\omega=0$) and, for every multi-index $I=(i_1<...<i_r)$,
$$ \partial_I^{p-1} (a_I):=
\left( \frac{\partial^r}{\partial z_I}\right)^{p-1}(a_I)=\frac{\partial^{r(p-1)}}{\partial z_I^{p-1}}(a_I)=\frac{\partial^{r(p-1)}}{\partial z_{i_1}^{p-1}...\partial z_{i_r}^{p-1}}(a_I)=0
$$
When $r=0$, a $0$-form is a rational function and $\omega\in R$ is $p$-closed when $\omega\in K$. 
\end{definition}

If $\eta=\sum_I b_i dz_I\in \Omega^r_{R/K}$, we write
$$ \partial_i (\eta):= \sum_I \partial_i (b_I)dz_I=
\sum_I \frac{\partial b_I}{\partial z_i} dz_I
$$

\begin{definition}\label{D:p-decomposable}
Let $K$ be a field of positive characteristic $p>0$. A {\it $p$-decomposable} $r$-form $\omega\in \Omega^r_{R/K}$ is defined inductively by the following conditions.
\begin{enumerate}
\item If $r=0$, that is, $\omega$ is a rational function, then we must have $\omega\in K$.
\item If $r>0$, then for every $i=1,...,n$, we require that
$$
\omega=dz_i\wedge(z_i^{p-1}\omega_i+\eta_i) +\tau_i
$$
where $\omega_i$, $dz_i\wedge z_i^{p-1}\omega_i$, $dz_i\wedge \eta_i+\tau_i$ are $p$-decomposable, $\partial_i^{p-1}(\eta_i)=0$ and $\tau_i$ is free of $dz_i$.
\end{enumerate}

\end{definition}

\begin{thm}
An $r$-form $\omega\in \Omega^r_{R/K}$, $r>0$, is $p$-closed if and only if, for every $i=1,...,n$,
$$
\omega=dz_i\wedge(z_i^{p-1}\omega_i+\eta_i) +\tau_i=z_i^{p-1}dz_i\wedge\omega_i+dz_i\wedge\eta_i +\tau_i
$$
where $\omega_i$, $z_i^{p-1}dz_i\wedge \omega_i$, $dz_i\wedge \eta_i+\tau_i$ are $p$-closed, $\partial_i^{p-1}(\eta_i)=0$ and $\tau_i$ is free of $dz_i$.
\end{thm}
\begin{proof}

First suppose that $\omega$ is a closed $r$-form such that, for each $i=1,...,n$, one can write
$$
\omega=dz_i\wedge (z_i^{p-1}\omega_i+\eta_i)+\tau_i=z_i^{p-1}dz_i\wedge \omega_i+dz_i\wedge\eta_i+\tau_i
$$
where $\omega_i$, $z^{p-1}dz_i\wedge \omega_i$, $dz_i\wedge\eta_i+\tau_i$ are $p$-closed, $\partial_i^{p-1}(\eta_i)=0$ and $\tau_i$ is free of $dz_i$. As sum of $p$-closed forms, then $\omega$ is $p$-closed.

Reciprocally, we prove that if $\omega$ is $p$-closed, then $\omega$ has the enunciated decomposition.

If $r=1$ and $\omega=\sum_{i=1}^n a_i dz_i$ is $p$-closed, then
$$
\omega=dz_i\wedge (z_i^{p-1} \omega_i+\eta_i)+\tau_i
$$
where $\omega_i=0$, $\eta_i=a_i$, $\tau_i=\sum_{j\neq i} a_j dz_j$, $ z_i^{p-1}dz_i\wedge \omega_i=0$, $dz_i\wedge \eta_i +\tau_i=\omega$ are $p$-closed,  $\partial_i^{p-1}(\eta_i)=\partial_i^{p-1}(0)=0$ and $\tau_i=0$ is free of $dz_i$.

Let $\omega$ be a $p$-closed $r$-form, $r> 1$. For each $i=1,...,n $, $\omega$ has the decomposition
$$
\omega=dz_i\wedge(z_i^{p-1}\omega_i+\eta_i) +\tau_i
$$
where ${\partial_i}(\omega_i)={\partial_i}^{p-1}(\eta_i)=0$ and $\tau_i$ is free of $dz_i$. We will prove that $ \omega_i$ is $p$-closed.

Since $d\omega=0$, we have
$$
0=d\omega=(-1)d z_i\wedge (z_i^{p-1}d\omega_i+d\eta_i)+d\tau_i
$$
and $d\tau_i=\partial_i(dz_i\wedge\rho_i)$ because $\tau_i$ has no term in $dz_i$. By the above equation we can write
$$
z_i^{p-1}d z_i\wedge d\omega_i=-dz_i\wedge d\eta_i+d\tau_i=-dz_i\wedge d\eta_i+\partial_i(dz_i\wedge\rho_i)
$$
As $\omega_i$ is free of $dz_i$, we can write $d\omega_i=dz_i\wedge \partial_i(\omega_i)+\xi_i$, where $\xi_i$ is free of $dz_i$. Since $\partial_i (\omega_i)=0$, we have that $d\omega_i=\xi_i$ is free of $dz_i$. If $d\omega_i\neq 0$, then also $d z_i\wedge d\omega_i=d z_i\wedge \xi_i\neq 0$ and, as $ (p-1)!\equiv (-1)  \mod  p$, then ${\partial_i}^{p-1}(z_i^{p-1}d z_i\wedge d\omega_i)= (-1)d z_i\wedge d\omega_i\neq 0$, but
$$
\partial_i^{p-1}(-dz_i\wedge d\eta_i+d\tau_i)=-dz_i\wedge d\partial_i^{p-1}(\eta_i)+\partial_i^{p-1}\partial_i(dz_i\wedge \rho_i)=0
$$
Therefore  $d\omega_i=0$, that is, $ \omega_i$ is closed. Now we will prove that $\omega_i$ is $p$-closed. We can write $\omega_i=\sum_J a_Jdz_J$, $\eta_i=\sum_J b_J dz_J$, where $ J=(j_1, \cdots ,j_r)$,  and 
$$
\omega=\sum_I a_I dz_I=dz_i\wedge (z_i^{p-1}\sum_J a_Jdz_J+\sum_J b_J dz_J)+\tau_i
$$
If $ i=i_m \in I=(i_1,...,i_{r})$  and $ J=I(i):=(i_1,\cdots, \widehat{i},..., i_{r})$, then
$$
a_I=(-1)^{m+1}(z_i^{p-1}a_{I(i)}+b_{I(i)})
$$
and, as $ \omega$ is $ p$-closed,
\[\begin{split}
0&=(\partial_I)^{p-1}(a_I)\\
&=(-1)^{m+1}(\partial_{I(i)})^{p-1} (\partial_i)^{p-1}(z_i^{p-1}a_{I(i)})+(\partial_{I(i)})^{p-1} (\partial_i)^{p-1}b_{I(i)}\\
&=(-1)^{m}(\partial_{I(i)})^{p-1}(a_{I(i)})
\end{split}\]
and $\omega_i$ is $p$-closed. Hence $z_i^{p-1}dz_i\wedge\omega_i$ and $dz_i\wedge \eta_i+\tau_i=\omega-z_i^{p-1}dz_i\wedge \omega_i$ are also $p$-closed.

\end{proof}

\begin{remark}
The set of $p$-closed $r$-forms  is a  $K(z^p)$-vector subspace of $\Omega^r_{K(z)/K}$.
\end{remark}

The definition \ref{D:p-closed} generalizes the definition of $p$-closed $1$-form given in \cite{Sweedler}, where a sketch of proof of the following theorem in the case $r=1$ can be found.

\begin{thm}[Positive Characteristic Poincaré Lemma]\label{T:Main}

Let $K$ be a field of positive characteristic $p>0$ and $ R$ be $ K[z]$ or $ K(z)$. An $r$-form $\omega$ is exact if and only if $\omega$ is $p$-closed.
\end{thm}
\begin{proof}
By Proposition \ref{preparation1} there exists  $\lambda \in K[z^p]$ such that $\lambda\cdot \omega \in K[z]$ and hence we can suppose that $R=K[z]$.

$(\Rightarrow)$ Suppose first that $\omega\in \Omega_{R/K}^r$ is exact, that is, $ \omega=d \eta$. The case $r=0$ is trivial. If $r>0$ and $\eta=\sum_J b_J dz_J$, then
$$
\omega=\sum_I a_I dz_I=\sum_{i=1}^n \sum_J \partial_i b_J dz_i\wedge dz_J
$$ 
If $I=(i_1<...<i_k<...<i_r)$ and $I(k):=(i_1<...< \widehat{i_k}<... <i_r)$, then
$$
a_I=\sum_{k=1}^r (-1)^k \partial_{i_k}b_{I(k)}
$$
Since $ \partial_{i_k}^{p-1}\partial_{i_k}=\partial_{i_k}^p=0 $ and $ |I|=r \geq 1 $, then $\partial_{I}^{p-1} (a_I)=0 $.


$ (\Leftarrow)$ Suppose that $\omega$ is $p$-closed. We proceed by induction on $n$ and $r$. For $r=0$ we have that $\omega=\lambda\in K$ and there is nothing to prove. Suppose the property true for $r$-forms, where $r\geq 0$. We now prove by induction on $n$ that if $\omega$ is a $p$-closed $(r+1)$-form, then  $\omega$ is exact. The case $n=0$ is trivial. Suppose the result true for $n\geq 0$ and let $\omega$ be a $(r+1)$-form over $K^{n+1}$. Since $\omega$ is $p$-closed, we can write
$$
\omega=dz_1\wedge(z_1^{p-1}\omega_1+\eta_1) +\tau_1=z_1^{p-1}dz_1\wedge \omega_1+dz_1\wedge \eta_1+\tau_1
$$
where $\omega_1$ is $p$-closed, $\partial_1^{p-1}(\eta_1)=0$ and $\tau_1$ is free of $dz_1$.

By induction hypothesis, $\omega_1$ is exact, that is, $\omega_1=d\alpha$.  Then $z_1^{p-1}dz_1\wedge \omega_1=d(-z_1^{p-1}dz_1 \wedge \alpha)$ is also exact, and hence  $z_1^{p-1}dz_1\wedge \omega_1$ is $p$-closed, by the first part of the proof. In this way, we have that $\omega-z_1^{p-1}dz_1\wedge \omega_1$ is $p$-closed and then we can suppose $\omega_1=0$. We define
$$
\theta=\int \eta_1 dz_1
$$
(in the obvious algebraic "integration" sense, which is possible because $(\frac{\partial}{\partial z_1})^{p-1}(\eta_1)=0$). Then $\theta\in \Omega^{r-1}_{R/K}$ and
$$
d\theta=dz_1\wedge \eta_1+\theta_1
$$
where $\theta_1$ does not involve $dz_1$.

Replacing $\omega$ by $\omega-d\theta=\tau_1-\theta_1$, we reduce to the case where $\omega$ is free of $dz_1$ and we can see it in $\Omega_{K(z_1^p)[z_2,...,z_{n+1}]/K(z_1^p)}^{r+1}$. By our induction hypothesis (over $n$) the $(r+1)$-form $\omega$ is $p$-closed as an element of $\Omega_{K(z_1^p)[z_2,...,z_{n+1}]/K(z_1^p)}^{r+1}$, and hence it is also $p$-closed as an element of $\Omega_{K[z_1,z_2,...,z_{n+1}]/K}^{r+1}$.

\end{proof}


\begin{cor}\label{corollary_P}
Let $\omega=\sum_{I=(i_1<...<i_r)} a_I(z) dz_I\in \Omega^r_{R/K}$ be a closed $r$-form on $ K^n$. If for each multi-index $I=(i_1,...,i_r)$ there is $s\in \{1,...,r\}$ such that
$$
\partial_{i_s}^{p-1}(a_I)=0 ,
$$
then $\omega$ is exact.
\end{cor}
\begin{proof}
The enunciated property imply immediately the $p$-closedness condition.

\end{proof}

\begin{remark}
The converse of the corollary is not true. In fact, if
$$
\omega=(x^{p-1}+y^{p-1})dx\wedge dy
$$
then
$$
\omega=d(xy^{p-1}dy-x^{p-1}ydx)
$$
but
$$
{\partial_x}^{p-1}(x^{p-1}+y^{p-1})={\partial_y}^{p-1}(x^{p-1}+y^{p-1})=1
$$
\end{remark}

\section{Decomposition of Closed Forms and the Inverse of Cartier Operator}\label{S:Decomposition}

Consider the $K(z)$-linear transformation
$$
\gamma_0: \Omega_{K(z)/K}^*\rightarrow H_{K(z)}^*
$$
$$
a_I(z)dz_I\mapsto a_I(z^p)z_I^{p-1}dz_I \mod B_{K(z)}^r
$$
where $a(z^p):=a(z_1^p,...,z_n^p)$, and $K(z)$ acts on $H_{K(z)}^*$ via $F_0:K(z)\rightarrow K(z)$, $a(z)\mapsto a(z^p)$. 

Using the Poincaré Lemma, we will prove in this section that $\gamma_0$ is an isomorphism of $K(z)$-vector spaces. When $K$ is a perfect field (in particular if $K$ is algebraically closed), we can take, instead of $\gamma_0$, the $K(z)$-linear transformation
$$
\gamma: \Omega_{K(z)/K}^*\rightarrow H_{K(z)}^*
$$
$$
a_Idz_I\mapsto a^pz_I^{p-1}dz_I \mod B_{K(z)}^r
$$
where $K(z)$ acts on $H_{K(z)}^*$ via the {\it Frobenius endomorphism} $F:K(z)\rightarrow K(z)$, $a\mapsto a^p$. It will be also an isomorphism, called the inverse of {\it Cartier Operator}, as mentioned in the next section,  where we show how the isomorphism $\gamma_0$ can be proved as a consequence of the isomorphism $\gamma$, and how it can be used to prove the Poincaré Lemma. 


We need the following two operators $\phi, \psi: \Omega_{K(z)/K}^r\rightarrow \Omega_{K(z)/K}^r$, given by
$$
\phi(\sum_I a_I dz_I)=\sum_I \partial_I^{p-1}(a_I) dz_I
$$
and
$$
\psi(\sum_I a_I dz_I)=\sum_I \partial_I^{p-1}(a_I)z_I^{p-1} dz_I
$$

We can now make the following definition.

\begin{definition}
We call $ \omega_\mathcal{I}:=\psi(\omega)$ the {\it irrational part} and $\omega_\mathcal{R}:= \omega - \omega_\mathcal{I}$ the {\it rational part} of the closed $ r$-form $ \omega$, giving the decomposition
$$
\omega = \omega_\mathcal{I}+\omega_\mathcal{R}
$$
\end{definition}

Let $I=(i_1<...<i_r)$ be an ordered multi-index. We write, by abuse of notation, $i_1\in I$,..., $i_r\in I$, and $j\notin I$ if $j\neq i_1$,..., $j\neq i_r$. Then
$$
I-i_s:=(i_1<...<\widehat{i_s}<...<i_r)
$$
when $i_s\in I$, and
$$
I+j:=(i_1<...<i_s<j<i_{s+1}<...<i_r)
$$
if $i_s<j<i_{s+1}$ (with the obvious adaptation to extremal situations).

\begin{lemma}
If $\omega\in \Omega_{K(z)/K}^r$ is closed, then $\partial_j\phi(\omega)=0$, for every $j=1$,..., $n$.
\end{lemma}
\begin{proof}
Since $\omega=\sum_I a_I dz_I$ is closed, we obtain, for $j\notin I$,
$$
\partial_j (a_I)=\sum_{i\in I} (-1)^{m_I(i,j)}\partial_i (a_{I-i+j})
$$
for suitable numbers $m_I(i,j)\in \mathbb N$. Therefore
$$
\partial_j \partial_I^{p-1}(a_I)=\partial_I\partial_j (a_I)=\sum_{i\in I} (-1)^{m_I(i,j)}\partial_I^{p-1}\partial_i (a_{I-i+j})=0
$$
If $j\in I$, then $\partial_i\partial_I^{p-1}(a_I)=\partial_j^p\partial_{I-j}^{p-1}(a_I)=0$.

\end{proof}

\begin{prop}\label{P:IrrationalPartClosed}
If $\omega=\sum_I a_I dz_I\in \Omega_{K(z)/K}^r$ is closed, then $\phi(\omega)$ is $p$-closed (hence exact) and $\omega_\mathcal{I}=\psi(\omega)$ is closed. More precisely, for every multi-index $I$, $\partial_I^{p-1}(a_I)\in K(z^p)$.
\end{prop}
\begin{proof}
Let $\omega=\sum_I a_I dz_I$ be a closed $r$-form. By the above lemma, $\partial_i\partial_I^{p-1}(a_I)=0$ for every $i=1$,..., $n$, which implies $\partial_I^{p-1}(a_I)\in K(z^p)$. Also $\phi(\omega)$ is closed because
$$
d\phi(\omega)=\sum_I\sum_{j\notin I}\partial_j \partial_I^{p-1}(a_I)dz_j\wedge dz_I=0
$$
and hence $\phi(\omega)$ is $p$-closed because
$$
\partial_I^{p-1} \partial_I^{p-1}(a_I)=0
$$
Finally $\omega_\mathcal{I}$ is closed because
$$
d\omega_\mathcal{I}=\sum_I\sum_{j\notin I}\partial_j(\partial_I^{p-1}(a_I)z_I^{p-1})dz_j\wedge dz_I=\sum_I\sum_{j\notin I}\partial_I^{p-1}(a_I)\partial_j (z_I^{p-1})dz_j\wedge dz_I=0
$$

\end{proof}

\begin{prop}
If $\omega=\sum_I a_I dz_I$ is closed, then $\omega_\mathcal{R}$ is $p$-closed (hence exact), while $\omega_\mathcal{I}$ is closed but it is $p$-closed only in the case $\omega_\mathcal{I}=0$. So a closed $r$-form is $p$-closed if and only if $\omega_\mathcal{I}=0$.
\end{prop} 
\begin{proof}
By the above proposition, we have that
$$
\omega_{\mathcal{I}}=\sum_I \partial_I^{p-1}(a_I) z_I^{p-1}dz_I=\sum_I b_I(z^p)z_I^{p-1}dz_I
$$
is closed, with $\partial_I^{p-1}(b_I(z^p)z_I^{p-1})=b_I(z^p)\partial_I^{p-1}(z_I^{p-1})=b_I(z^p)$. Also
$$
\omega_{\mathcal{R}}=\sum_I (a_I-\partial_I^{p-1}(a_I) z_I^{p-1})dz_I=\sum_I (a_I-b_I(z^p)z_I^{p-1})dz_I
$$
is closed, with $\partial_I^{p-1}(a_I-b_I(z^p)z_I^{p-1})=\partial_I^{p-1}(a_I)-b_I(z^p)=b_I(z^p)-b_I(z^p)=0$. Therefore, $\omega_{\mathcal{R}}$ is $p$-closed. 

\end{proof}

\begin{thm}\label{T:Main2}
$\gamma_0: \Omega_{K(z)/K}^*\rightarrow H_{K(z)}^*$ is an isomorphism.

\end{thm}
\begin{proof}
Let $\omega=\sum_I a_I dz_I\in Z^r$ be a closed $r$-form. By Proposition \ref{P:IrrationalPartClosed}, $\partial_I^{p-1}(a_I)\in K(z^p)$, then
$$
\omega=\gamma_0(\sum_I F_0^{-1}(\partial_I^{p-1}(a_I))dz_I)
$$
and $\gamma_0$ is surjective.

We also have
$$
\gamma_0(\omega)=\sum_I a_I(z^p)z_I^{p-1}dz_I \mod B^r
$$

If $\gamma_0(\omega)=0$, then $\sum_I a_I(z^p)z_I^{p-1}dz_I$ is $p$-closed, which implies $a_I=0$ for every $I$, so $\omega=0$ and hence $\gamma_0$ is injective.

\end{proof}


\section{Deduction of Poincaré's Lemma from the Inverse of Cartier Operator}\label{S:Cartier}

In this section we follows closely the Chapter $1$ of book \cite{Brion}.

Let $A$ be a $K$-algebra (in general, $A=K[z]$ or $A=K(z)$). We have the {\it derivation operator}
$$
d:\Omega_{A/K}^*\rightarrow \Omega_{A/K}^*
$$
$$
d(a_1da_2\wedge ... \wedge da_r)=da_1\wedge da_2 \wedge ... \wedge da_r
$$

In particular, we have the derivation $d:A\rightarrow \Omega_{A/K}^1$. We define de {\it ring of $\partial$-constants}
$$
A_0:=\{a\in A; d(a)=0\}
$$
Then obviously $K\subset A_0$. If $p$ is the characteristic of $K$, then
$$
A^p:=\{a^p; a\in A\} \subset A_0
$$
(where if $p=0$, then $A^p=K$.)

The $d$-operator is $A_0$-linear and satisfies $d^2=0$, hence we obtain a complex $(\Omega_A^*,d)$, called the {\it De Rham complex} of $A$ (over $K$).

\begin{lemma}[\cite{Brion}, 1.3.2 Lemma, p. 23]
Let $A$ be an algebra over an algebraically closed field $K$. The space $Z_A^*:=\{\alpha\in\Omega_{A/K}; d\alpha=0\}$ is a graded $A^p$-subalgebra of $\Omega_A^*$, and the space $B_A^*:=\{d\alpha; \alpha\in \Omega_{A/K}^{*-1}\}$ is a graded ideal of $Z_A^*$. Hence, the quotient space
$$
H_A^*:=Z_A^*/B_A^*
$$
is a graded-commutative $A^p$-algebra. If also $A$ is a localization of a finitely generated algebra, then the $A$-module $\Omega_{A/K}^*$ is finitely generated, and the $A^p$-modules $Z_A^*$, $B_A^*$ and $H_A^*$ are finitely generated as well.

\end{lemma}

Consider the map $\gamma: A\rightarrow \Omega_{A/K}^1$, $\gamma(a)=a^{p-1}da$.

\begin{lemma}[\cite{Brion}, 1.3.3 Lemma, p. 24]
With the above notations, for all $a, b \in A$, we have:
\begin{enumerate}
\item $\gamma(ab)=a^p\gamma(b)+b^p\gamma(a)$.
\item $d\gamma(a)=0$.
\item $\gamma(a+b)-\gamma(a)-\gamma(b)\in B_A^1$.

\end{enumerate}
\end{lemma}

Now consider the composition
$$
A\overset{\gamma}\longrightarrow Z_A^1\longrightarrow Z_A^1/B_A^1=H_A^1
$$
which is a $K$-derivation and $A$ acts on itself by multiplication, and on $H_A^1$ via the {it Frobenius endomorphism} $F:A\rightarrow A$, $F(a)=a^p$. This derivation will be also denoted by $\gamma$. The universal property of Kähler differentials yields an $A$-algebra homomorphism
$$
\gamma: \Omega_{A/K}^*\rightarrow H_A^*
$$
$$
a_1da_2\wedge ...\wedge da_r\mapsto a_1^pa_2^{p-1}... a_r^{p-1}da_2\wedge ... \wedge da_r \mod B_A^*
$$
where $A$ acts on $H_A^*$ via the Frobenius endomorphism.

\begin{thm}[\cite{Brion}, 1.3.4 Theorem, p. 24]
If $K$ is algebraically closed and $A$ is a regular ring, then $\gamma: \Omega_{A/K}^*\rightarrow H_A^*$ is an isomorphism.

\end{thm}

The inverse of the isomorphism $\gamma$ on the above theorem is called the {\it Cartier Operator}. In particular, suppose $K$ algebraically closed and $A=K(z)=K(z_1,...,z_n)$. Then
$$
\gamma(a(z)dz_I)=a(z)^p\cdot z_I^{p-1}dz_I \mod B_A^{\mid I\mid}
$$
and every closed $r$-form $\omega$ can be write as
$$
\omega=\sum_I a_Idz_I=\sum_I b_I^p z_I^{p-1}dz_I+d\eta=\sum_I \left(b_I^p z_I^{p-1}+\sum_{i\in I}(-1)^{m(I,i)}\frac{\partial c_{(I-i)}}{\partial z_{i}}\right) dz_I
$$
for suitable numbers $m(I,i)\in \mathbb N$. Hence
$$
a_I=b_I^p z_I^{p-1}+\sum_{i\in I}(-1)^{m(I,i)}\frac{\partial c_{(I-i)}}{\partial z_{i}}
$$
so that
$$
\left( \frac{\partial^{\mid I\mid}}{\partial z_I}\right)^{p-1}(a_I)=b_I^p
$$
This imply that $\omega=d\eta$ if and only if $\omega$ is $p$-closed.

In the general case where $K$ is not necessarily algebraically closed, we can consider the algebraic closure $\bar{K}$ of $K$ and then we obtain an inclusion $K(z)\hookrightarrow \bar{K}(z)$, which enable us to make the same argument to conclude the $p$-closedness condition for exactness. Also the $K$-linear transformation
$$
\gamma_0: \Omega_{K(z)/K}^*\rightarrow H_{K(z)}^*
$$
$$
a(z)dz_{i_1}\wedge ...\wedge dz_{i_r}\mapsto a(z^p)z_{i_1}^{p-1}... z_{i_r}^{p-1}dz_{i_1}\wedge ...\wedge dz_{i_r} \mod B_{K(z)}^r
$$
where $a(z^p):=a(z_1^p,...,z_n^p)$, is an isomorphism of $K(z)$-vector spaces, if $K(z)$ acts on $H_{K(z)}^*$ via $F_0:K(z)\rightarrow K(z^p)$, $a(z)\mapsto a(z^p)$.

\section{Rational and irrational parts  of   closed forms}

In this section we describe how to split a polynomial or rational closed $ r$-form into a $p$-closed and a non $p$-closed part, and how these parts   generate the  algebraic de Rham $r$-cohomology over the ring $R$ of the polynomials $K[z]$ or over the field $ K(z)$.

In this section we denote by $z=(z_1,\cdots ,z_n) \in K^n$  and
$$
(z_i\partial_i)^{p-1}\eta:= z_i^{p-1}\partial_i^{p-1} \eta 
 $$ 
 
The following  proposition presented by Sweedler \cite{Sweedler} is a particular case of the splitting in the case of an 1-form.

\begin{prop}\label{prop2.1}
Let $ K$ be a field of characteristic  $ p>0$ and $ \omega= \sum_{i=1}^n a_i(z) dz_i$ be a polynomial or rational $1$-form on $K^n$. Suppose that $ \omega$ is closed    $ $ and define  the decomposition $ \omega = \omega_{\mathcal R} + \omega_{\mathcal I}$ such that
 $$ \omega_I = \sum_{i=1}^n b_i dz_i , \qquad  b_i =(-1) (z_i\partial_ i)^{p-1} a_i , \qquad  \omega_{\mathcal R} = \omega - \omega_{\mathcal I}$$
Then this decomposition  is such that  
\begin{itemize}
\item[a)]$ \omega_{\mathcal R}$ is  $ p$-closed (hence exact) and is called the {\bf rational part} of the integral of $\omega$.

\item[b)]$\omega_{\mathcal I}$ is closed and $ \omega_{\mathcal I}=  \sum_{i=1}^n (a_i z_i^p) \frac{dz_i}{z_i}$ is called the {\bf irrational part} of the integral of $ \omega$.
\end{itemize}
\end{prop}
 
\begin{proof}
As  $ \omega$ is closed and if $ i \neq j = 1,2,...,n$ then $  \partial_i a_j- \partial_j a_i=0$ and 
$$
(-1)\partial_i b_j= \partial_i ( z_j\partial_j)^{p-1} a_j= ( z_j\partial_j)^{p-1} (\partial_i a_j) =( z_j\partial_j)^{p-1} (\partial_j a_i)=
z_j (\partial_j)^p a_i=0
$$
and then  $ \partial_i b_j= \partial_j b_i=0$, that is,   $ \omega_I$ is closed. As $ \omega $ is closed then $ \omega_R=\omega-\omega_I$ is also closed. As  $(p-1)!\equiv (-1)$ mod $p$ and $ \partial_i^p=0$ therefore
    $$\partial_i^{p-1} b_i=(-1) \partial_i^{p-1}( z_i^{p-1} \partial_i^{p-1}a_i)= 
    (-1)(p-1)!\partial_i^{p-1}a_i=\partial_i^{p-1}a_i$$ 
 and then 
 $$\partial_i^{p-1} (\omega_R)_i=\partial_i^{p-1} (\omega- \omega_I)_i= \partial_i^{p-1}a_i -\partial_i^{p-1}b_i =\partial_i^{p-1}a_i -\partial_i^{p-1}a_i=0 $$ 
 By  definition \ref{D:p-closed} $ \omega_R$ is $ p$-closed and one could conclude applying the Poincaré Lemma \ref{T:Main3} that  $ \omega$  is exact but we will follow the reasoning sketched in \cite{Sweedler}.
 
If $ \omega_R \in \Omega_{K(z)/K}$  is rational $ r$-form  then, by \ref{preparation1}, there is a  $ \partial$-constant $ q^p$ such that    $ q^p \omega$ is a polynomial $ r$-form  and one can suppose that $ \omega_R=\sum_{i=1}^n c_i$ is a polynomial 1-form and $ c_i$ is free of $ z_i^{p-1}$ for each $ i=1,\cdots ,n$. As $ \omega_R$ is closed $ \partial_j c_i=\partial_i c_j $ and then $ \partial_j\int  c_i dz_i= c_j$  and then $\gamma = (1/n)\sum_{i=1}^n \int c_i dz_i$ is such that $  \partial_i\gamma= c_i$  for $ i=1, \cdots ,n$ and then $ d \gamma= \omega$, that is, $\omega_R$ is exact.  
\end{proof}

\begin{definition}\label{operator_P}
Let $ K $ be a field of characteristic $p>0$, $ r$  an integer $ 1\leq r \leq n $, $J=\{i_1<...<i_r\}$ an ordered subset of $ \{1,2,..., n\}$ and $\mathcal{S}$ be the set of all ordered  subsets  of $ \{1,2,..., n\}$ with $r$ elements. If $ \omega \in \Omega^r_{R/K}$ is an algebraic $r$-form and $\omega=\sum_J a_J dz_J$, where $ J \in \mathcal{S}$, then we define:

\begin{itemize}
 \item[\textbf{a)}] The operator $\mathcal{P}_J:R\rightarrow R $
$$
\mathcal{P}_J := \prod_{i \in J} z_i^{p-1} \partial_i^{p-1}=(z_{i_1}\cdots z_{i_r})^{p-1} \partial_{i_1}^{p-1}\cdots \partial_{i_1}^{p-1}
$$
With this notation  $\mathcal P_{(i)}=\mathcal{P}_ i=z_i^{p-1}\partial_i^{p-1} $, 
  $ \mathcal{P}_J=\prod_{i \in J} \mathcal{P}_i$ and $ \mathcal{P}_i\mathcal{P}_j=\mathcal{P}_j\mathcal{P}_i$.
 
 \item[\textbf{b)}] The operator $ \mathcal{Q}_r : \Omega^r_{R/K(z)} \rightarrow \Omega^r_{R/K(z)}$ such that, for $\omega=\sum_{J} a_J dz_J$,  
$$
\mathcal{Q}_r(\omega)= (-1)^r \sum_{J \in \mathcal{S}} \mathcal{P}_J (a_J) dz_J
$$ 

 \item[\textbf{c)}] The decomposition $ \omega = \omega_\mathcal{R}+\omega_\mathcal{I}$ such that  $ \omega_\mathcal{I}=\mathcal{Q}_r(\omega)$  is the {\it irrational part} and $ \omega_\mathcal{R}= \omega - \omega_\mathcal{I}$ is the {\it rational part} of the $ r$-form $ \omega$.
  \end{itemize}
 
\end{definition}

\begin{prop}\label{properties of P}
The operators $ \mathcal{P}_J$ and $\mathcal Q_r$ have the following properties.

\begin{itemize}
\item[(i)] If $ i \in \{ 1,2, \cdots,n\}$, then $ \mathcal{P}_i \mathcal{P}_i= -\mathcal{P}_i$ and 
$\mathcal{P}_{\{i<j\}}=\mathcal{P}_i\mathcal{P}_j$.

\item[(ii)] If $ J \in \mathcal{S} $ and $ i\in J$, then $ \mathcal{P}_i\partial_i=0 $ and $\mathcal{P}_i \mathcal{P}_J=-\mathcal{P}_J$ .

\item[(iii)] If $ J \in \mathcal{S} $ and $|J|=r $, then  $ \mathcal{P}_J\mathcal{P}_J=(-1)^{r}\mathcal{P}_J$ and $ \mathcal{Q}_r\mathcal{Q}_r= \mathcal{Q}_r$ .

\item[(iv)] $\mathcal{Q}_r(\omega_\mathcal{R})=0 $ and if $ \omega$ is closed then $ \omega_\mathcal{I}$ is closed   and $ \omega_\mathcal{R}$ is  $p$-closed. 

\end{itemize}

\end{prop}

\begin{proof}
As $ (p-1)!\equiv -1\mod  p$ and $ \partial_i^q=0$ if $ q \geq p$, then 
\item[(i)]
\[\begin{split}
\mathcal{P}_i\mathcal{P}_i&=z_i^{p-1}\partial_i^{p-1}z_i^{p-1}\partial_i^{p-1}\\
&=z_i^{p-1}(-1)\partial_i^{p-1}\\
&=-\mathcal{P}_i
\end{split}\]

If $i<j$, we have
\[\begin{split}
\mathcal{P}_i\mathcal{P}_j&=z_i^{p-1}\partial_i^{p-1}z_j^{p-1}\partial_j^{p-1}\\
&=z_i^{p-1}z_j^{p-1}\partial_i^{p-1}\partial_j^{p-1}\\
&=\mathcal{P}_{\{i<j\}}
\end{split}\]

      
  \item[(ii)] Without loss of generality one can suppose that $ J=( 1,2,...,r)$ and $ i=r$, then 
$$
\mathcal{P}_J\partial_r = z_1^{p-1} \partial_1^p...z_r^{p-1} \partial_r^{p-1}\partial_r=z_1^{p-1} \partial_1^p...z_r^{p-1} \partial_r^{p}=0
$$
and, now supposing $i=1$,
\[\begin{split}
\mathcal{P}_1\mathcal{P}_J&=z_1^{p-1}\partial_1^{p-1}\mathcal{P}_J\\
&=z_1^{p-1}\partial_1^{p-1}\prod_{i \in J} \mathcal{P}_i\\
&=z_1^{p-1}\partial_1^{p-1}z_1^{p-1}\partial_1^{p-1}z_2^{p-2}\partial_2^{p-1}...z_r^{p-1}\partial_r^{p-1}\\
&=z_1^{p-1}\partial_1^{p-1}(z_1^{p-1})\partial_1^{p-1}z_2^{p-2}\partial_2^{p-1}...z_r^{p-1}\partial_r^{p-1}+z_1^{2(p-1)}\partial_1^{2(p-1)}z_2^{p-2}\partial_2^{p-1}...z_r^{p-1}\partial_r^{p-1}\\
&=z_1^{p-1}(-1)\partial_1^{p-1}z_2^{p-2}\partial_2^{p-1}...z_r^{p-1}\partial_r^{p-1}+0\\
&=-\mathcal{P}_J
\end{split}\]

 
  
\item[(iii)] As $ \partial_i \partial_j= \partial_j \partial_i$, then 
$$
\mathcal{P}_J \mathcal{P}_J= ( \prod_{i\in J}\mathcal{P}_i )\mathcal{P}_J=(-1)^{|J|}\mathcal{P}_J
$$ 
and
\[\begin{split}
\mathcal{Q}_r \mathcal{Q}_r(\omega)&=(-1)^{2r}\sum_{J \in \mathcal{S}} (\mathcal{P}_J)^2 (a_J) dz_J\\
&=(-1)^{|J|}\sum_{J \in \mathcal{S} } \mathcal{P}_J (a_J) dz_J\\
&=(-1)^{r}\sum_{J \in \mathcal{S} } \mathcal{P}_J (a_J) dz_J\\
&=\mathcal{Q}_r (\omega)
\end{split}\]

 
\item[(iv)] $\mathcal{Q}_r (\omega_\mathcal{R})=\mathcal{Q}_r (\omega) - \mathcal{Q}_r (\omega_\mathcal{I})=\mathcal{Q}_r (\omega) - \mathcal{Q}_r^2(\omega)=\mathcal{Q}_r (\omega)-\mathcal{Q}_r(\omega)=0$
  
Let $ \mathcal{S}_r$ be all the ordered subsets of size $r$ of $ \{ 1, 2,...,n\}$ and $\mathcal{S}_{r+1}$ be all the ordered subsets of size $r+1$ of $ \{ 1, 2,...,n\}$. If $ \omega=\sum_{J \in \mathcal{S}_r } a_J dz_J$, then $d \omega= \sum_{I \in \mathcal{S}_{r+1}} b_I dz_I $. If $I=\{ i_1,i_2,\cdots,i_{r+1} \} $  and $ I(i_k)= \{ i_1,i_2, \cdots , \widehat{i_k},\cdots ,i_{r+1} \}$, then 
$$
b_I= \sum _{k=1}^{r+1} (-1)^{k-1}\partial_{i_k} a_{I(i_k)}
$$
If $ d\omega=0$,  then  $ b_I=0$ and, for each $ k=1,..., r+1$, one can isolate $ a_{I(i_{k})}$ on the left hand of $ b_I=0$ such that 
  $$  (-1)^{k-1}\partial_{i_k} a_{I(i_k)}= \sum_{j=1,j \neq k}^{r+1} (-1)^j \partial_{i_j}a_{I(i_j)} \qquad (*)$$
Note that $ \{ I(i_k): k=1,...,r+1 , J\in \mathcal{S}_{r+1} \} = \mathcal{S}_r$, that is, each $ I(i_k)$ is such that  $I(i_k)=J$ for some $J \in \mathcal{S}_r$ and reciprocally. Therefore
\[\begin{split}
d(\omega_\mathcal{I})&=d(\sum_{J \in \mathcal{I}} \mathcal{P}_J (a_J) dz_J )\\
&=\sum_{J \in \mathcal{I}} \sum_{i\notin \mathcal{J}} \mathcal{P}_J \partial_i(a_J) dz_i\wedge dz_J \\
&=\sum_{J \in \mathcal{I}} \sum_{i\notin \mathcal{J}} \sum_{j\in \mathcal{J}}(-1)^{R_{ij}}\mathcal{P}_J \partial_j(a_{(J-\{j\})\cup \{i\}}) dz_i\wedge dz_J \\
&=0
\end{split}\]
and then $ \omega_\mathcal{I}$ is closed and   $ \omega_\mathcal{R}= \omega - \omega_\mathcal{I}$ is also closed.  
  
If we suppose that $ \omega_\mathcal{R} =\sum_{J \in \mathcal{I}} c_J dx_J $ is not $ p$-closed, then there is some $ J \in \mathcal{I}$ such that $ \mathcal{P}_J c_I \neq 0$, and then $\mathcal{Q}_r(\omega_\mathcal{R})\neq 0$, which is a contradiction.
\end{proof}


We summarize the previous discussions about $ p$-closedness of an $ r-$form on $ K^n$

\begin{thm}\label{equivalence}
If  $\omega$ is an $r$-form and $\omega\in \Omega^r_{R/K}$, then the following conditions are equivalent
\begin{itemize}
\item[(a)] $\omega$ is $ p$-closed;
\item[(b)] $\omega$ is  $ p$-decomposable;
\item[(c)] $\omega$ is exact;
\item[(d)] $\omega$ is closed and $\mathcal{Q}_r( \omega)=0$. 
\end{itemize}
\end{thm}

From Proposition \ref{properties of P} ( see \cite{Warner}) follows the

 \begin{thm} \label{deRham}
  Let $ K $ be a field of characteristic $p>0$ and  $ R$  be the ring of polynomial $ K(z)$ or the field of rational functions $ K(z)$,  $ \omega$ be a closed polynomial or rational  $ r$-form  with  $ 1\leq r \leq n $. Then 
  \begin{itemize}  
  
 \item[a)] The algebraic de Rham cohomology $ H^r(K^n,R)$ over the ring $R$ is generated by the class of the irrational forms $ \omega_I$ with  $\omega \in \Omega^r(K^n) $ and the chain sequence defined by exterior derivation $ d_{r-1} : \Omega^{r-1}(K^n)\rightarrow  \Omega^r(K^n) $ is not exact.
 
 \item[b)] The algebraic rational de Rham  cohomology  defined by the  $ p$-closed $ r$-forms $ \Omega_p^r$ is such that $ H_p^r(K^n,R)=0$ and the chain sequence defined by exterior derivation  on closed and  $ p$-closed $ r$-forms  
 $ d_{r-1} : \Omega_p^{r-1}(K^n)\rightarrow  \Omega_p^r(K^n) $ is exact. 
\end{itemize}  
   
 \end{thm}

\section{Completely integrable  closed r-forms }
 
 A polynomial $ p$-closed $ r$-form in $ K^n$ may present some powers $ x_i^{p-1}$ and in this case it is not possible to integrate $ \int z_i^{p-1} dz_i$ . In this  section we  define an operator  that splits a polynomial  $ r$-form  $ \omega$ in $ K^n$ into two parts such that one of them is free of powers $ z_i^{p-1}$ and we prove that if $ \omega$ is closed then these parts of the decomposition are also closed. 

\begin{definition}\label{operator_O}
With the same hypothesis  and notations of definition \ref{operator_P} we define 
\begin{itemize}
 \item[\textbf{a)}] the operator   $ \mathcal{O}_J : K[z] \rightarrow K[z]$
$$ \mathcal{O}_J= \prod_{i \in J} (1- (z_i\partial_i)^{p-1}) -1;$$
 where $ 1=id$ denotes the identity operator. With this notation  $\mathcal{O}_{\{ i\}}=(z_i\partial_i)^{p-1} $
 
 \item[\textbf{b)}] the operator  $ \mathcal{O}_r : \Omega^r(K^n) \rightarrow \Omega^r(K^n)$    such that  
$$ \mathcal{O}_r(\omega)= (-1) \sum_{J \in \mathcal{I}} \mathcal{O}_J a_J dz_J;$$ 

 \item[\textbf{c)}]the decomposition $ \omega = \omega_C+\omega_T$ such that 
  $ \omega_T=\mathcal{O}_r(\omega)$  is the \textit{ restrict} part, that is, with integrability restriction in at least one variable,   and $ \omega_C= \omega - \omega_T$ is the \textit{completely integrable} part of the $ r$-form $ \omega$.
  \end{itemize}
 
\end{definition}

\begin{example}
 \begin{itemize}
 \item[a)] 
 If $ r=2$,  $ n=3 $ and $ \omega = a_{12}dz_1\wedge dz_2+a_{13}dz_1\wedge dz_3 +a_{23}dz_2\wedge dz_3$ 
  then the restrict part of $ \omega_T$ is  given by
 $ \omega_T=b_{12}dz_1\wedge dz_2+b_{13}dx_1\wedge dz_3 +b_{23}dz_2\wedge dz_3 $ such that 
$$ (-1) b_{ij}= (z_i\partial_i)^{p-1} a_{ij}+ (z_j\partial_j)^{p-1} a_{ij}+
 (z_i\partial_i)^{p-1} (z_j\partial_j)^{p-1} a_{ij}$$
 
  \item[b)]
 If $ r=3$, $ n=3 $ and $ \omega = a_{123}dz_1\wedge dz_2 \wedge dz_3 $ 
 then the restrict part of $ \omega$ is given by
 $ \omega_T= b_{123} dz_1\wedge dz_2\wedge dz_3$ such that
 $$ (-1) b_{123}= \{  (z_1\partial_1)^{p-1}+ (z_2\partial_2)^{p-1}+ (z_3\partial_3)^{p-1}+$$
     $$(z_1 z_2\partial_1 \partial_2)^{p-1}+ (z_1 z_3\partial_1 \partial_3)^{p-1}+
     (z_2 z_3\partial_2 \partial_3)^{p-1}+$$
     $$(z_1 z_2 z_3\partial_1 \partial_2 \partial_3)^{p-1}     
  \} a_{123}$$ 
In this case $ (-1)  b_{123}= \{\prod_{i=1}^3 (1- (z_i\partial_i)^{p-1})-1\} a_{123}$.
\end{itemize}
\end{example}

The following two lemmas are a preparation to  Proposition \ref{decomp}.
 
\begin{lemma}\label{propriedades}
In the conditions and notations of the above definition  :
\begin{itemize}
\item[\textbf{a)}] If $ i \in I$ then
 $$ \partial_i^{p-1}\mathcal{O}_I= (-1)\partial_i^{p-1} ;$$
 
 \item[\textbf{b)}] If $ k=1,\cdots ,n$ then
 $$  \partial_k \mathcal{O}_J=  \mathcal{O}_J \partial_k $$
 
 \item[\textbf{c)}] 
 $$ \mathcal{O}_J \mathcal{O}_J =-\mathcal{O}_J  \quad \mathrm{and} \quad 
 \mathcal{O}_r \mathcal{O}_r =-\mathcal{O}_r  $$ 
 
  \item[\textbf{d)}] If we denote $ \mathcal{O}_{\{i\}}= \mathcal{O}_i$ and if $ \mathcal{J}$ denotes all non void subsets of $J= \{ 1,2, \cdots ,r \}$ and $ J_q$ is the set of all sets of $\mathcal{J}$   with $ q$ elements then
  $$ \mathcal{O}_J= \sum_{q=1}^r \sum_{ \quad\{i_1,\cdots,i_q \} \in J_q} \mathcal{O}_{i_1} \cdots \mathcal{O}_{i_q}  $$
and we call the above expression as the expanded polynomial form of  the operator $ \mathcal{O}_J$
  
\end{itemize}  
 
 \end{lemma}

\begin{proof}

\textbf{a)} The proof follows by induction on the number $ r$ . If $r=1$ the assertion is true as  in the proof of \ref{prop2.1},
that is, if $ \omega= \sum_{i=1}^n a_i(z) dz_i$ then  for any $ i=1,\cdots ,n$
$$ \partial_i ^{p-1} \mathcal{O}_{\{i\}} a_i=\partial_i^{p-1} b_i= (-1)\partial_i^{p-1}a_i$$

Suppose, by induction hypothesis, that if $ r>1$ and $J$ is any ordered subset of $ \{1, \cdots ,n \}$ with $ r-1$ elements
 then  for any $ j \in J$
 $$ \partial_j^{p-1}\mathcal{O}_J= (-1)\partial_j^{p-1} .$$
 If $ I$ has $ r$ elements and $ i \in I$ then  for any $ k \in I $ and $ k \neq i$

 $$
 \partial_i^{p-1}\mathcal{O}_I =\partial_i^{p-1}) \{ (1-(z_k \partial_k)^{p-1} 
   \left( \prod_{j \in I-k} (1- (z_j\partial_j)^{p-1}) -1+1\right)  -1  \} = 
 $$
$$  
    (1-(z_k \partial_k)^{p-1})\quad  \partial_i^{p-1} \left(\mathcal{O}_{I-k} +1 \right)- \partial_i^{p-1}
$$  
by induction hypothesis of 
$$ 
  \partial_i^{p-1}\mathcal{O}_I=(1-(z_k \partial_k)^{p-1}) (\quad(-1)\partial_i^{p-1} +\partial_i^{p-1})- \partial_i^{p-1}= (-1) \partial_i^{p-1}
 $$
   
   \textbf{b)} As all the the partial derivation commutes then 
    $  \partial_k \mathcal{O}_J=  \mathcal{O}_J \partial_k $.
   
   \textbf{c)} As in the proof of item a)  
   $$\mathcal{O}_{\{i\}} \mathcal{O}_\{i\}=( z_i \partial_i) ^{p-1} \mathcal{O}_{\{i\}}=(-1)( z_i \partial_i) ^{p-1}=(-1)\mathcal{O}_{\{i\}}$$
   then  $$ \left((1+(z_i \partial_i)^{p-1})\right)^2=(1+\mathcal{O}_{\{i\}})^2=(1+\mathcal{O}_{\{i\}}) $$
   by the definition of $ \mathcal{O}_J $
   $$  \mathcal{O}_J ^2= \left(  \prod_{i \in J} (1-\mathcal{O}_{\{i\} })  -1  \right)^2=(-1)\mathcal{O}_J $$
   and by the definition of $ \mathcal{O}_r$ we have that $\mathcal{O}_r \mathcal{O}_r=(-1)\mathcal{O}_r$.
   
   \textbf{d)} If $\mathcal{O}_i=  (x_i \partial_i)^{p-1}$ then expanding the product in the definition of $ O_{J} $ we have
   $$ \mathcal{O}_J= \prod_{i \in J} (1- \mathcal{O}_{\{i\}}) -1=
   \sum_{q=1}^r \sum_{ \quad\{i_1,\cdots,i_q \} \in J_q} \mathcal{O}_{i_1} \cdots \mathcal{O}_{i_q} $$
   
\end{proof}

 Let $ K $ be a field of characteristic $p>0$, $ \omega$ a polynomial or rational $r$-form in $ \Omega^r(K^n)$ and $ 1\leq r < n $.
  If $ J=\{j_1<j_2< \cdots <j_r\}$ is an ordered subset of $\{1,2,\cdots ,n\} $ and $\mathcal{I}$ is the set of all ordered subsets  
  of $ \{1,2, \cdots , n\}$ then $ \omega=\sum_{J \in\mathcal{I} } a_J dx_J$. If $\mathcal{G}$ is the set o all ordered subsets of 
  $\{1,2,\cdots ,n\} $  with $ r+1$ elements,  $ G=\{ i_1<i_2< \cdots <i_{r+1} \}$, 
   $ dx_G=dx_{i_1}\wedge \cdots \wedge dx_{i_{r+1}}$, $ G(k)= i_1,\cdots \widehat{i_k} \cdots i_{r+1}$,    $ \sigma_k =( i_k, i_1,\cdots \widehat{i_k},\cdots i_{r+1})$ is a permutation and  $ sign(\sigma)=k-1 $ and  $ d\omega$ is written as 
 $$
 d\omega= \sum_{G \in \mathcal{G} }\left( \sum_{k=1}^{r+1} (-1)^{k-1 }\partial_{i_k} a_{G(k)} \right) dz_G 
 $$ 
 \begin{lemma}\label{notacao}
\begin{itemize} 
 \item[\textbf{a)}] For every $  i=1,2,\cdots ,n$  \qquad $ \mathcal{O}_i \partial_i=0 $  and 
  $$  \mathcal{O}_{ i_1,\cdots {i_k} \cdots i_{r+1} } \partial_{i_k}=\mathcal{O}_{ i_1,\cdots \widehat{i_k} \cdots i_{r+1}} \partial_{i_k}$$. 
 \item[\textbf{b)}] If $ \omega=0 $ then for every $G \in \mathcal{G} $  then  \qquad $\sum_{k=1}^{r+1} (-1)^{k-1 }\partial_k a_{G(k)}  =0$ and for each
  $1 \leq k \leq r+1 $
 $$(-1)^{k-1}\partial_k a_{G(k)}= \sum_{i \neq k, i=1} ^{r+1}  (-1)^{i}\partial_{j_i} a_{G(j_i)} \qquad (*)$$ . 
  \item[\textbf{c)}] If the the  irrational part of $\omega$ is  $ \omega _I = \sum_{J \in \mathcal{I}}  \mathcal{O}_J a_J dz_J$ then
$$ 
d (\omega_I)= \sum_{G \in \mathcal{G}} b_{G} dz_G =\sum_{G \in \mathcal{G}} \left( \sum_{k=1}^{r+1} \mathcal{O}_{G(k)} (-1)^{k-1 }\partial_{i_k} a_{G(k)} \right) dz_G \qquad (**) 
$$ .  
     \end{itemize} 
 
\end{lemma}
 
 \begin{proof}  
 
 \textbf{a)}
  $\partial_i \mathcal{O}_i=\partial_{i}(z_i\partial_i)^{p-1}=z_i \partial_i^p=0$ By definition \ref{operator_O} the partial derivation $ \partial_{i_k}$ commutes with $ \partial_{i_j}$  and the result follows from distribution  and elimination of $ (z_{i_k}\partial_{i_k})$.
 
 \textbf{b)} If $ d\omega=0$ then each coefficient of $ dx_{G}$ is zero and in equation (*) the term $(-1)^{k-1}\partial_k a_{G(k)} $ is isolated in the left hand of the equation. 
 
 \textbf{c)}By the item (b) of \ref{notacao} the operator $ \mathcal{O}_G(k)$ commutes with $ \partial_k$.
 \end{proof}

 \begin{prop}\label{decomp} 
 Let $ K $ be a field of characteristic $p>0$ and $ \omega$ be a polynomial or rational  $ r$-form  with  $ 1\leq r <n $. 
 \begin{itemize}
 \item[\textbf{a)}] 
   If $ \omega$ is a closed  $ r$-form then $ \omega_T$ and $ \omega_C$ are closed.    
  \item[\textbf{b)}] The decomposition $ \omega=\omega_C+ \omega_T $ is unique such that $ \mathcal{O}_r (\omega_C)=0$ and $ \mathcal{O}_r( \omega_T)=- \omega_T$.
   \item[\textbf{c)}]  If $ d_{r-1} : \Omega^{r-1}(K^n) \rightarrow \Omega^r(K^n)$ is the exterior derivation then
 $$ \mathcal{O}_r d_{r-1}=0 $$ 
 \end{itemize}
 \end{prop} 
     
 \begin{proof}
 
 \textbf{a)}
 The expression of $d(\omega_i) $ is given in the item (c) of the previous Lemma. Note that  for each coefficient of  $ G \in \mathcal{G} $ of $d(\omega_i) $ the index $ i_k$ of $ \partial_{i_k}$ determines uniquely the permutation
  $ \sigma_k =( i_k, i_1,\cdots \widehat{i_k},\cdots i_{r+1})$ and the set $ G(k)$, therefore there is no ambiguity in denoting
   only $ \partial_{i_k}$ instead of $ \partial_{i_k} a_{G(k)}$. We will prove that  for each $ G \in \mathcal{G}$ the coefficient 
   $$ b_G =\sum_{k=1}^{r+1} \mathcal{O}_{G(k)} (-1)^{k-1 }\partial_{i_k} a_{G(k)}=0$$.
   
   Without loss of generality one can suppose that $ G=\{1,2,\cdots ,r\}$ and then $ \partial_{i_k}=\partial_k$,     $ G(k)=\{1,2, \cdots, \hat{k} \cdots ,r+1\}$ and 
    $$ b_G= \sum_{k=1}^{r+1} \mathcal{O}_{G(k)} (-1)^{k-1 }\partial_{k}  $$
    Substitute  the term $(-1)^{k-1 }\partial_{k} $ by the  right hand of the  equation (*) of item (b) of Lemma \ref{notacao}, and so we have
     $$ b_G= \sum_{k=1}^n \mathcal{O}_{G(k)} \sum_{i \neq k, i=1}^{r+1} (-1)^i \partial_{i}$$    
   By the item (a) of \ref{notacao} if  $i \neq k $ then $\mathcal{O}_{G(k)} \partial_i = \mathcal{O}_{G(i,k)}\partial_i$  where $ G(i,k)=G-\{ i,k \}$ and then 
     $$ b_G= \sum_{k=1}^n  \sum_{i \neq k, i=1}^{r+1} (-1)^i \mathcal{O}_{G(k,i)}\partial_{i}$$
for each $ k=1,\cdots ,n$  collect the coefficients of $ \partial_k $  such that
$$
 b_G= \sum_{k=1}^n  (\sum_{i \neq k, i=1}^{r+1} \mathcal{O}_{G(i,k)}) (-1)^k\partial_{k}
 $$
Now we replace the operator $\mathcal{O}_{G(i,k)} $ by its expanded  polynomial expression as in item d) of the Lemma \ref{propriedades}
$$ b_G= \sum_{k=1}^n  \left(\sum_{i \neq k, i=1}^{r+1} \left(  
        \sum_{q=1}^{r-2} \sum_{ \quad\{j_1,\cdots,j_q \} \in G(i,k)_q} \mathcal{O}_{j_1} \cdots \mathcal{O}_{j_q}               
                                                             \right)
                     \right) (-1)^k\partial_{k} =
  $$                   
   $$                 
  = \sum_{k=1}^n  H(k) (-1)^k\partial_k $$
where $G(i,k)_q$ denotes all the non void subsets of $ G(i,k)$ with $ q$ elements. Each coefficient $ H(k)$ of $(-1)^k \partial_k $ is a sum of product of operators where the operator $ \mathcal{O}_k$ does not appear 
$$ H(k)= (n-2)( 
\mathcal{O}_1+\cdots +\widehat{\mathcal{O}_k} \cdots+\mathcal{O}_n 
           ) +
           (n-3)(   
  \mathcal{O}_1 \mathcal{O}_2+\cdots +
  \widehat{\mathcal{O}_k \mathcal{O}_n}+  \cdots +\mathcal{O}_{n-1}\mathcal{O}_n          
           ) +
$$
      $$ +\cdots    + 
        \sum_{i \neq k,i=1}^{r+1} \mathcal{O}_1 \mathcal{O}_2 \cdots \widehat{\mathcal{O}_i} \cdots   \widehat{\mathcal{O}_k} \cdots \mathcal{O}_{r+1}     
      $$
We will conclude that $ d(\omega_I)$ is closed  distributing  the product of $ (-1)^k \partial_k$ to the above sum and factoring each operator that appears in  $ H(k)$. Note that for each $ G(k)$ with $ k=1,2,\cdots n$, as in item b) of the Lemmma \ref{notacao}, is such that  
$$ \partial_1 + \cdots +(-1)^{k-1}+\cdots + (-1)^{r}\partial_{r+1}=0$$ 
. Then, for each $ i=1,2, \dots ,r+1$
$$  \mathcal{O}_i \left((-1)\partial_1 + \cdots \widehat{\partial_i} \cdots
 +(-1)^{n+1}\partial_{r+1}\right)=  \mathcal{O}_i  (-1)^{i+1}\partial_i=0 $$ 
 For each $ i \neq j =1,2, \cdots ,r+1$
 $$ \mathcal{O}_i \mathcal{O}_j \left(\partial_1 + \cdots + \widehat{(-1)^{i-1}\partial_i} \cdots+ \widehat {(-1)^{j-1}\partial_j}+\cdots+
 +(-1)^{n}\partial_{r+1}\right)= 
 $$
 
 $$=\mathcal{O}_i \mathcal{O}_j( (-1)^{i}\partial_i+(-1)^j \partial_j)=0
   $$
and
$$
\mathcal{O}_1 \mathcal{O}_2 \cdots \widehat{\mathcal{O}_i} \cdots   \widehat{\mathcal{O}_k} \cdots \mathcal{O}_{r+1}( (-1)^i \partial_i)=
$$
   $$
  = \mathcal{O}_1 \mathcal{O}_2 \cdots \widehat{\mathcal{O}_i} \cdots   \widehat{\mathcal{O}_k} \cdots \mathcal{O}_{r+1}
    \left((\partial_1 + \cdots + \widehat{(-1)^{i-1}\partial_i}+\cdots+(-1)^{n}\partial_{n+1} \right)=0
   $$
 and then $ d(\omega_I)=0$. As $ \omega $ and $\omega_I$   are closed then the rational part  $ \omega_R$ is also closed. 
 
 \textbf{b)}  By definition \ref{operator_O} 
  $ \mathcal{O}_r(\omega)= (-1) \sum_{J \in \mathcal{I}} \mathcal{O}_J a_J dz_J$ and by Lemma \ref{propriedades} a) 
  $ \partial_i^{p-1}\mathcal{O}_I= (-1)\partial_i^{p-1} $ and  if $ \omega= \sum_{J \in \mathcal{I}} a_J dz_J$ and if $ i \in J$ then
  $ (\partial_i)^{p-1} a_J= (-1) O_J a_J$ and
   $$(\partial_i)^{p-1} (\omega_R)_J= (\partial_i)^{p-1} (a_J - \mathcal{O}_J (a_J))=0$$
   By Corollary \ref{corollary_P} $\omega_R $  is exact and then it is  $p$-closed. 

 \textbf{c)} As $ d_{r-1} \omega$ is an exact $ r$-form then, by Poincaré Lemma \ref{T:Main3}, $ \omega$ is $ p$-closed  and  
   $ (d_{r-1}\omega)_T=\mathcal{O}_r d_{r-1}=0 $
 \end{proof}
 
\bibliographystyle{amsplain}

\end{document}